\documentclass[12pt]{amsart}

\usepackage{graphicx}%
\usepackage{multirow}%
\usepackage{amsmath,amssymb,amsfonts}%
\usepackage{amsthm}%
\usepackage{mathrsfs}%
\usepackage{hyperref}
\usepackage[english]{babel}
\usepackage{enumitem}
\usepackage{float}
\usepackage{lipsum}
\usepackage[foot]{amsaddr}
\usepackage{csquotes}

\usepackage{tikz}
\newcommand*\wtildea[1]{
  \begingroup
    \settowidth{\dimen0}{$#1$}
    \rlap{\resizebox{\dimen0}{\totalheight}{$\widetilde{\phantom{x\vphantom{#1}}}$}}%
  \endgroup
  #1}

\usepackage{comment}
\usepackage{bbm}

\usepackage[doi=false,isbn=false,url=false,eprint=false, giveninits=true, maxbibnames=99]{biblatex}

\newcommand{\tint}{\int_{\mathbb{T}}}
\newcommand{\bbt}{\mathbb{T}}
\newcommand{\pw}{\mathbb{P}(\omega)}
\newcommand{\infzt}{\inf_{z\in\mathbb{T}}}

\newcommand{\var}{\mathrm{var}}

\newtheorem{theorem}{Theorem}
\numberwithin{theorem}{section}

\newtheorem{proposition}[theorem]{Proposition}%
\newtheorem{definition}[theorem]{Definition}%
\newtheorem{remark}[theorem]{Remark}%

\newtheorem{lemma}[theorem]{Lemma}
\newtheorem{corollary}[theorem]{Corollary}

\newtheorem{example}[theorem]{Example}%

\usepackage{enumitem}
\setenumerate[1]{label={(\arabic*)}}

\DeclareUnicodeCharacter{0306}{{\u i}}

\title[Average metric entropy for random Blaschke products]{Average measure theoretic entropy for a family of expanding on average random Blaschke products}

\author{Cecilia González-Tokman}
\email{cecilia.gt@uq.edu.au}
\address{School of Mathematics and Physics, The University of Queensland, St Lucia, QLD 4072, Australia}
\author{Renee Oldfield}
\email{renee.oldfield@uq.edu.au}

\usepackage[dvipsnames]{xcolor}
\newcommand{\com}[1]{\textcolor{black}{#1}}

\begin{document}
\begin{abstract}
This work gives a computable formula for the average measure theoretic entropy of a family of expanding on average random Blaschke products, generalizing \cite{expandPujals} to the random setting. In doing so, we describe the random invariant measure and associated measure theoretic entropy for a class of admissible random Blaschke products, allowing for maps which are not necessarily expanding and may even have an attracting fixed point.
\end{abstract}

\maketitle

\section{Introduction}

In the study of chaotic dynamical systems, it is natural to consider measure theoretic entropy, also known as metric entropy or Kolmogorov-Sinai entropy, a valuable measure theoretic invariant representing the complexity of the dynamical system. The notion of measure theoretic entropy is developed by Kolmogorov in \cite{og_m.e} and by Sinai in \cite{sinai_ent}. In general, the measure theoretic entropy is challenging to compute for a dynamical system, or even determine if it is zero or positive \cite{zero_pos_entropy}.

Measure theoretic entropy shares a connection with the positive Lyapunov exponents of a map. 
Pesin proves in \cite{pesin_1977} that for a diffeomorphism on a compact manifold, the measure theoretic entropy with respect to an ergodic probability measure absolutely continuous with respect to Lebesgue is equal to the sum of the positive Lyapunov exponents. In \cite{ruelle_entropy}, Ruelle proves the Margulis-Ruelle inequality, which asserts that for a continuously differentiable map the measure theoretic entropy is bounded above by the sum of the positive Lyapunov exponents. Ledrappier and Young give a formula for the metric entropy of a diffeomorphism in terms of its Lyapunov exponents in \cite{LY1,LY2}. In \cite{walters_inner}, Walters proves a convergence theorem for the Perron-Frobenius operators of a family of distance-expanding maps and uses this to investigate the invariant measures and equilibrium states, the measures of maximal (measure theoretic) entropy, of these maps. 

In \cite{bogenschutz_1992}, Bogenschütz shows that some properties of the measure-theoretic entropy for a map in the deterministic case extend to the random setting. As a corollary to the more general result, he proves a variational principle for a random dynamical system. Kifer establishes the existence and uniqueness of equilibrium states for smooth, expanding random maps in \cite{kifer_expand}. Urbański and Simmons investigate the equilibrium states of random distance expanding maps and establish a variational principle in \cite{SU2014}. In \cite{AFGV}, Atnip, Froyland, González-Tokman, and Vaienti establish a variational principle and the existence of relative equilibrium states for random open and closed interval maps, in addition to proving a random Ruelle-Perron-Frobenius theorem. 

Obtaining results for systems that are expanding on average, rather than uniformly expanding, enables the investigation of a much broader class of systems. Kifer and Khanin extend the results from \cite{kifer_expand} to the case where the maps are expanding on average in \cite{KK_average}. Buzzi establishes the existence of absolutely continuous SRB measures for random expanding on average Lasota-Yorke maps under certain conditions in \cite{buzzi_SRB} and the uniqueness of such a measure under a covering condition in \cite{buzzi_edoc}. Mayer, Skorulski, and Urbański establish a Ruelle-Perron-Frobenius theorem for expanding random maps and use inducing on first return times to show that these results extend to random expanding on average cocycles \cite{urbanski}. 
Froyland, González-Tokman and Murray establish the stability of eventually expanding on average interval map cocycles under a variety of perturbations, including Ulam's method, in \cite{FGTR_average}. Atnip, Froyland, González-Tokman and Vaienti develop a quenched thermodynamic formalism for random dynamical systems that satisfy a random covering condition in \cite{AFGV2}.

A particular class of functions of interest are Blaschke products. These functions, comprised of products of Möbius transformations mapping the unit circle to itself, are amenable to rigorous analysis. As such, Blaschke products are a natural choice for investigation. They have connections to operator theory, hyperbolic geometry, and dynamics. The various applications and connections of Blaschke products are discussed extensively in \cite{Blaschke_connect,Blaschke_applications}. In \cite{Martin_expanding}, Martin obtains a sufficient condition on the zeros of a Blaschke product for it to be uniformly expanding on the unit circle, and describes the measure theoretic entropy for these Blaschke products. The entropy for expanding Blaschke products is also investigated in \cite{Blaschke_graph}. Bandtlow, Just, and Slipantschuk explicitly determine the spectrum of Perron-Frobenius operators for these uniformly expanding Blaschke products in \cite{Slipantschuk_2017}.

In \cite{GTQ}, González-Tokman and Quas completely characterize the Lyapunov spectrum of the Perron-Frobenius operator cocycle of a uniformly expanding random Blaschke product cocycle, generalizing \cite{Slipantschuk_2017} to the random setting. They also obtain necessary and sufficient conditions for the stability of its spectrum, showing that small perturbations can lead to a collapse of the spectrum. Blaschke product cocycles fixing the origin are investigated by González-Tokman and Peters in \cite{JPGT}. 

Pujals, Robert, and Shub describe the metric entropy for a general Blaschke product and prove a lower bound for the average measure theoretic entropy of a family of Blaschke products consisting of a fixed Blaschke product $T$ composed with rotations, that is, the family $(\theta T)_{\theta\in\mathbb{T}}$ \cite{expandPujals}. 

This paper aims to extend the results of \cite{expandPujals} to the random setting.  This would allow us to obtain information on the measure theoretic entropy with only knowledge of the maps, rather than $\omega$-dependent orbits, in particular, without needing to explicitly find the random invariant measure for each cocycle.

Before stating our results, we first define a transfer operator. Let $C(\mathbb{T})$ denote the space of continuous functions on the complex unit circle $\mathbb{T}.$
\begin{definition}[Transfer operator]\label{transfer_operator}
     Let $(X, \mathcal{B}, m)$ be a measure space and $T:X\to X$ a nonsingular transformation.
     The unique operator $\mathcal{L}_{T}:L^{1}(m)\to L^{1}(m)$ satisfying $$\int_{A}\mathcal{L}_{T}f\, dm=\int_{T^{-1}(A)}f\,dm$$ for all $f\in L^{1}(m)$ and $A\in\mathcal{B}$ is called the transfer operator (associated to $T$).
\end{definition}

For our first main result, we \com{describe the random invariant measure} of an admissible Blaschke product cocycle (Definition \ref{admissible_BP}). We consider a probability space $(\Omega, \mathcal{F},\mathbb{P})$ with invertible, ergodic, $\mathbb{P}$-preserving transformation $\sigma:\Omega\to\Omega$ and a family of maps $\mathcal{T}:=(T_{\omega})_{\omega\in\Omega}.$ For each $\omega\in\Omega,$ let $\mathcal{L}_{\omega}:=\mathcal{L}_{T_{\omega}}.$

\begin{theorem}\label{buzzi_blaschke}

    Let $\sigma:\Omega\to\Omega$ be an invertible, ergodic, measure-preserving transformation of a probability space $(\Omega, \mathcal{F}, \mathbb{P}).$ Let $(\mathcal{T}, \sigma)$ be an admissible Blaschke product cocycle. Let $m$ denote the Lebesgue measure.
    Then, there exists a unique random absolutely continuous invariant measure (acim) $\mu,$ with disintegration $\{\mu_{\omega}\}_{\omega\in\Omega}$ with respect to $\mathbb{P},$ such that \begin{enumerate}[label=(\roman*)]
        \item There exists a measurable map $x:\Omega\to D$ such that for $\mathbb{P}$-a.e. $\omega\in\Omega,$ $x_{\omega}=\lim_{n\to\infty}T_{\sigma^{-n}\omega}^{(n)}(z)$ for all $z\in D,$ and $h_{\omega}:=\frac{d\mu_{\omega}}{dm}=P_{x_{\omega}}  (z)=\frac{1-|x_{\omega}|^2}{|z-x_{\omega}|^2}, z\in\mathbb{T}.$ Moreover, $\mathcal{L}_{\omega}h_{\omega}=h_{\sigma\omega}$ for $\mathbb{P}$-a.e. $\omega\in\Omega;$ and \label{main_thm_rfp} 
        \item For $\mathbb{P}$-a.e. $\omega\in\Omega$ and any $h_{*}$ such that $||h_{*}||_{1}=1$ and $h_{*}$ is of bounded variation, there exists $n_{0}(\omega)\in\mathbb{N}$ and $\rho<1$ such that for all $n\ge n_{0}(\omega),$ $||\mathcal{L}_{\sigma^{-n}\omega}^{(n)}h_{*}-h_{\omega}||_{\infty}\le\rho^n.$ \label{main_thm_diff_bound}

    \end{enumerate}
\end{theorem}

For our second main result, we use Theorem \ref{buzzi_blaschke} to obtain a formula for the average measure theoretic entropy for a 1-parameter family of admissible Blaschke product cocycles, extending \cite{expandPujals} to the expanding on average random setting.

\begin{theorem}[Average metric entropy of an admissible Blaschke product cocycle]\label{main_thm2} Let $(\mathcal{T}_{1}:=\mathcal{T}, \sigma)$ be as in Theorem \ref{buzzi_blaschke}.
     For fixed $\theta\in\mathbb{T},$ let $\mathcal{T}_{\theta}:=(\theta T_{\omega})_{\omega\in\Omega}.$ Let $h_{\mathbb{P}}(\sigma)$ be the measure theoretic entropy of the base dynamical system with respect to $\mathbb{P}.$ Then, the average metric entropy over $\theta\in\mathbb{T}$ of the family of cocycles $(\mathcal{T}_{\theta})_{\theta\in\mathbb{T}}$, $\bar{h}(\mathcal{T}):=\tint h^{fib}_{\mu_{\theta}}(\mathcal{T}_{\theta})d\theta+h_{\mathbb{P}}(\sigma),$ is given by $$\bar{h}(\mathcal{T})=\int_{\Omega}\int_{\mathbb{T}}\log|T_{\omega}'|\,dm \,d\mathbb{P}(\omega)+h_{\mathbb{P}}(\sigma),$$ where $\mu_{\theta}$ is the random invariant measure for $\mathcal{T}_{\theta}$ as in Theorem \ref{buzzi_blaschke}.
\end{theorem}
\com{The main steps in extending the work of Pujals, Robert and
Shub from the deterministic setting to the random setting are (i) establishing
the existence and properties of the random invariant measure, and (ii) adapting the harmonic
trick of \cite{expandPujals} to the random setting.
The first step is achieved by relying on a result of Buzzi (Theorem 2.21) to establish the existence of a
random invariant measure, and by exploiting properties of Blaschke products to
describe what this invariant measure must be (Theorem \ref{buzzi_blaschke}). 
For the second step, we rely on Fubini’s and dominated convergence theorems, as well as on properties of the Poisson kernel.}

In Section \ref{prelim} we give relevant definitions and results that will be used throughout the paper. In Section \ref{entropy_section}, we use the results of Buzzi \cite{buzzi_SRB, buzzi_edoc} to prove Theorem \ref{buzzi_blaschke}. In Section \ref{average_section}, we use Theorem \ref{buzzi_blaschke} and the techniques of \cite{expandPujals} to prove Theorem \ref{main_thm2}.

In Section \ref{examples} we give an example of an admissible Blaschke product cocycle and numerically approximate the measure theoretic entropy of $\mathcal{T}_{\theta}$ for two different drivings and compare the numerically approximated average entropy over $\theta\in\mathbb{T}$ to the analytically computed value. By considering a Blaschke product cocycle fixing the origin, we also show that the conditions on admissible Blaschke product cocycles are sufficient but not necessary for the conclusions of Theorem \ref{main_thm2} to hold. 

\section{Preliminaries}\label{prelim}

\subsection{Finite Blaschke products and their acims}

In this section, we give relevant definitions and results about Blaschke products, their acims, and their measure theoretic entropy.

Let $D=\{z\in\mathbb{C}:|z|<1\}$ denote the complex unit disc and $\mathbb{T}=\{z\in\mathbb{C}:|z|=1\}$ the complex unit circle.
\begin{definition}[Inner function]
     Let $f:D\to\mathbb{C}$ be a bounded, analytic function, and $f^{*}(z):=\lim_{r\nearrow1}f(rz),$ provided this limit exists. We call $f$ an inner function if $|f^*(z)|=1$ for a.e. $z\in\mathbb{T}.$
\end{definition}

One class of inner functions of interest for this work is Blaschke products, holomorphic functions mapping the unit circle/disc to itself, given by the following definition.
\begin{definition}[Finite Blaschke product]\label{blaschke_defn}
    Let $\hat{\mathbb{C}}=\mathbb{C}\cup\{\infty\}.$ Fix $n\ge1.$ A map $T:\hat{\mathbb{C}}\to\hat{\mathbb{C}}$ given by $$T(z)=\theta_{0}\prod_{i=1}^{n}\frac{z-a_{i}}{1-\bar{a}_{i}z},$$ where $\theta_{0}\in\bbt$ and $a_{i}\in D,$ is called a degree$-n$ Blaschke product, or simply a (finite) Blaschke product.

\end{definition}

\begin{proposition}[{\cite[Proposition 1]{Martin_expanding}}]\label{derivative_bound}
For $z\in\mathbb{T},$ we have $$|T'(z)|=\sum_{i=1}^{n}\frac{1-|a_{i}|^2}{|z-a_{i}|^2}>0.$$
\end{proposition}
    
 Let $Q:[0,1)\to\mathbb{T}$ be given by $Q(t)=e^{2\pi it}.$ Throughout, we identify $T:\mathbb{T}\to\mathbb{T}$ with $S:[0,1)\to[0,1)$ where $S$ is given by $S(t)=(Q^{-1}TQ)(t).$ 
Figure \ref{fig:interval} shows an example of one such map.

\begin{figure}
    \centering
    \includegraphics[width=0.6\linewidth]{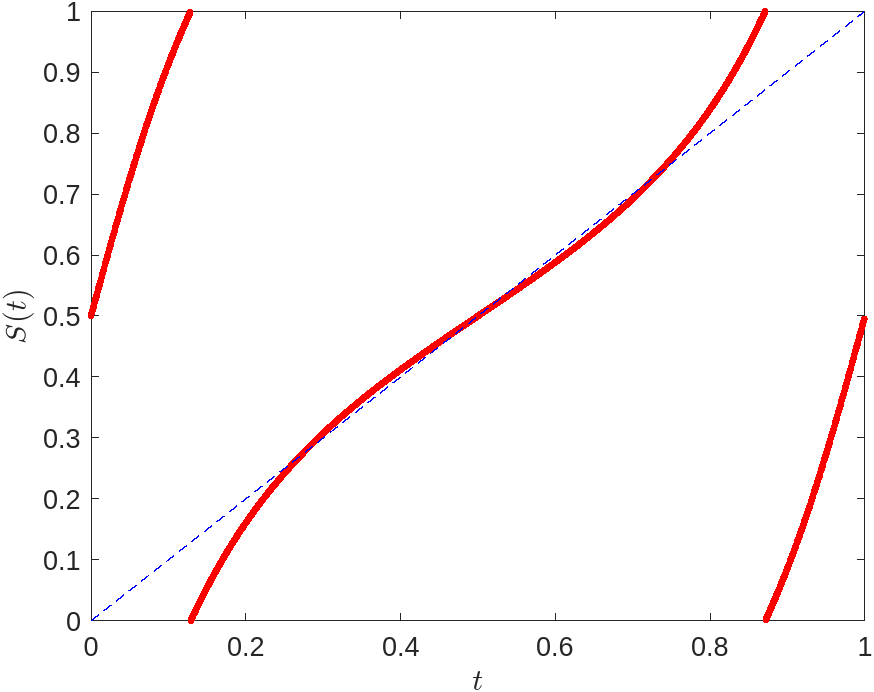}
    \caption{$t$ vs $S(t)=(Q^{-1}TQ)(t)$ where $T(z)=-\left(\frac{z-0.4}{1-0.4z}\right)^2, Q(t)=e^{2\pi i t}.$}
    \label{fig:interval}
\end{figure}

\begin{remark}
    When $n=1,$ $T$ is a Möbius transformation, and so some literature restricts the definition of a Blaschke product to $n\ge2,$ but throughout this work we consider $n\ge1.$
\end{remark}

\begin{theorem}[{\cite[Proposition 1.2]{expandPujals}}]\label{f.p.}
     For $n\ge2,$ let $T$ be a degree$-n$ Blaschke product. Then one of three mutually exclusive cases holds.
    \begin{enumerate}
        \item $T$ has $n+1$ fixed points on $\mathbb{T},$ where one is attracting and the remaining $n$ are expanding.
        \item $T$ has $n-1$ expanding fixed points on $\bbt.$ $T$ has one attracting fixed point $x\in D$ and another fixed point at $1/\bar{x}.$
        \item $T$ has $n+1$ fixed points on $\bbt$ where one fixed point $x$ is indifferent, that is, $|T'(x)|=1.$
    \end{enumerate}
\end{theorem}

\begin{proposition}[{\cite[Proposition 1.4]{topics_complex}}]\label{harmonic}
    For $x\in D$ and a continuous function $f:\mathbb{T}\to\mathbb{C},$ let $\tilde{f}:D\to\mathbb{C}$ be given by $$\tilde{f}(x)=\int_{\mathbb{T}}f(z)\frac{1-|x|^{2}}{|x-z|^{2}}\,dm(z).$$ We call $\tilde{f}$ the harmonic extension of $f$ to the disc and $P_{x}:=P_{x}(z)=\frac{1-|x|^{2}}{|z-x|^{2}}$ the Poisson kernel (associated to $x$).
    
    $\tilde{f}$ has the following properties:
    \begin{itemize}
    \item $\tilde{f}$ is harmonic on $D;$ and
        \item For $z\in\mathbb{T},$  $\lim_{r\nearrow1}\tilde{f}(rz)=f(z).$
    \end{itemize}
    
\end{proposition}

The next lemma follows from \cite[Lemma 8.1.3]{complex_analysis_book} \com{and will be used in the proof of Lemma \ref{average}.}
\begin{lemma}\label{harmonic_holomorphic}
    If $\tilde{f}:D\to D$ is harmonic and $T:D\to D$ is holomorphic, then $\tilde{f}\circ T$ is harmonic.
\end{lemma}

The next theorem follows from \cite[Proposition 4.1]{expandPujals}, and we include a proof here for completeness.
\begin{theorem}\label{poisson_aut} If $T$ is a Blaschke product, $f\in C(\mathbb{T})$ and $x\in D,$ then $$\int_{\mathbb{T}}f\circ T\cdot P_{x}\,dm=\int_{\mathbb{T}}f\cdot P_{T(x)}\,dm.$$
\end{theorem}

\begin{proof}
    Let $f\in C(\mathbb{T}).$ Fix $x\in D.$ Then we have \begin{align*}
        \tint f\circ T\cdot P_{x}\,dm&=\widetilde{f\circ T}(x)\\&=\tilde{f}\circ T(x)\\&=\tint f\cdot P_{T(x)}\,dm,
    \end{align*} where the first and third lines follow from Proposition \ref{harmonic} and the second line from Lemma \ref{harmonic_holomorphic}.
\end{proof}
For a function $f:\mathbb{T}\to\mathbb{C},$ we let $||f||_{\infty}:=\text{ess}\sup_{z\in\mathbb{T}}|f(z)|.$
\com{Lemmas \ref{l_infty} and \ref{poisson_diff} will be used in the proof of Theorem \ref{buzzi_blaschke}$(i).$}
\begin{lemma}\label{l_infty}
    Let $0<r<1$ and take $z_{0}\in\mathbb{T}.$ Then,
    $$\lim_{r\to1}||P_{rz_{0}}||_{\infty}=\infty.$$
\end{lemma}
\begin{proof}
    Since $P_{rz_0}(z)=\frac{1-|rz_{0}|^2}{|z-rz_{0}|^2}=\frac{1-r^2}{|z-rz_{0}|^2}$ is continuous on $\mathbb{T},$ it follows that \begin{align*}||P_{rz_{0}}||_{\infty}&=\frac{1-r^2}{\inf_{z\in\mathbb{T}}|z-rz_{0}|^2}\\&=\frac{1-r^2}{|z_{0}-rz_{0}|^2}\\&=\frac{1+r}{1-r}.\end{align*} In the second line we have used the fact that $z_{0}$ is the closest point on $\mathbb{T}$ to $rz_{0}.$ Thus the claim follows. 
\end{proof}

\begin{lemma}\label{poisson_diff}
    Take $x,y\in D.$ Then, $$\left|\left|P_{x}-P_{y}\right|\right|_{\infty}\le\frac{12|x-y|}{(1-|x|)^2(1-|y|)^2}.$$
\end{lemma}
\begin{proof}
For $z\in\mathbb{T}$ we have
    \begin{align*} 
        \left|P_{x}(z)-P_{y}(z)\right|&=\left|\frac{1-|x|^2}{|z-x|^2}-\frac{1-|y|^2}{|z-y|^2}\right|\\&\le\frac{(1-|x|^2)||z-y|^2-|z-x|^2|+|z-x|^2||y|^2-|x|^2|}{|z-x|^2|z-y|^2}\\&\le\frac{(1-|x|^2)|(z-y)^2-(z-x)^2|+|z-x|^2|y^2-x^2|}{|z-x|^2|z-y|^2}\\&=\frac{(1-|x|^2)|y-x|(|y+x-2z|+|z-x|^2|y+x|)}{|z-x|^2|z-y|^2}\\&\le\frac{12|x-y|}{(1-|x|)^2(1-|y|)^2},  
        \end{align*} where the second, third and last lines follow by applying the triangle inequality.
\end{proof}

\begin{definition}\label{measure_x}
    Consider a point $x\in\overline{D}.$ We define the probability measure $\mu_{x}$ as follows.
    \begin{itemize}
        \item If $x\in D,$ then let $\mu_{x}$ be the measure satisfying $\frac{d\mu_{x}}{dm}=P_{x}$ where $m$ denotes the Lebesgue measure.
        \item If $x\in\mathbb{T},$ let $\mu_{x}=\delta_{x},$ the Dirac measure.
    \end{itemize}
\end{definition}

\begin{remark}[\cite{expandPujals}]
     Let $\phi_{x}$ be the Möbius transformation  mapping 0 to $x,$ $\phi_{x}(z)=\frac{z+x}{1+\bar{x}z}$. For $x\in D,$ $\mu_{x}$ in Definition \ref{measure_x} is equivalently given by $\mu_{x}=(\phi_{x})_*(m),$ the pushforward of $m$ with respect to $\phi_{x}$.
    
\end{remark}

If $x\in\overline{D}$ is a fixed point of a Blaschke product $T,$ then it follows from Definition \ref{measure_x} and Theorem \ref{poisson_aut} that $\mu_{x}$ is $T$-invariant, that is, $(T)_{*}(\mu_{x})=\mu_{x}.$

The next theorem follows as a consequence of \cite[Theorem 1.2]{expandPujals}.
\begin{theorem} \label{det_entropy}
    Let $T$ be a Blaschke product of degree $n\ge2$ with fixed point $x\in D$. Let $\mu_{x}$ be as in Definition \ref{measure_x}. The measure-theoretic entropy of $T$ with respect to $\mu_{x}$ is given by $$h_{\mu_{x}}(T)=\int_{\mathbb{T}}\log|T'(z)|\,d\mu_{x}.$$ 
\end{theorem}

\subsection{Map cocycles and their measure theoretic entropy}
In this section, we give relevant definitions and results relating to map cocycles and their measure theoretic entropy that will be used throughout this work.

\begin{definition}[Map cocycle]
    Let $(\Omega,\mathcal{F},\mathbb{P})$ be a probability space. Let $\mathcal{T}:=(T_{\omega})_{\omega\in\Omega}$ be a family of measurable transformations of $(X,\mathcal{B}).$ A map cocycle is a tuple $\mathcal{R}=(\mathcal{T},\sigma)$ where $\sigma:\Omega\to\Omega$ is an invertible, ergodic, measure-preserving transformation of $(\Omega,\mathcal{F},\mathbb{P}).$ We call $(\Omega, \mathcal{F},\mathbb{P}, \sigma)$ the base dynamical system. We will simply write $\mathcal{T}$ when $\sigma$ is clear from the context. 
\end{definition}

Throughout we will refer to a \textit{$\mathcal{T}$-invariant measure} (or \textit{random invariant measure}), given by the follow definition.
\begin{definition}[$\mathcal{T}$-invariant measure]\label{}
    Let $(\mathcal{T},\sigma)$ be a map cocycle. Let $\mu$ be a measure on $\Omega\times\mathbb{T}$ which has disintegration $\{\mu_{\omega}\}_{\omega\in\Omega}$ with respect to $\mathbb{P}.$ We call $\mu$ a $\mathcal{T}$-invariant measure if $\mu_{\omega}$ is a probability measure, $\omega\mapsto\mu_{\omega}$ is measurable and for $\mathbb{P}$-a.e. $\omega\in\Omega$ $$(T_{\omega})_{*}(\mu_{\omega})=\mu_{\sigma\omega}.$$
\end{definition}

The following result due to Bogenschütz, adapted to our setting, tells us that the metric entropy of a random dynamical system is given as the sum of its so-called fibre entropy and the entropy of its base dynamical system.
\begin{lemma}[{\cite[Corollary 1]{bogenschutz_AR}}]\label{skew}

    Let $(\Omega,\mathcal{F},\mathbb{P},\sigma)$ be a dynamical system where $\sigma:\Omega\to\Omega$ is an ergodic invertible $\mathbb{P}$-preserving transformation. Let $(\mathcal{T}, \sigma)$ be a map cocycle. Furthermore, let $\mu$ be a $\mathcal{T}$-invariant measure. For a partition $\mathcal{P}$ of $X,$ let $$h^{fib}_{\mu}(\mathcal{T};\mathcal{P})=\lim_{n\to\infty}\frac{1}{n}\int_{\Omega} H_{\mu_{\omega}}\left(\bigvee_{i=0}^{n-1}T^{-i}_{\omega}\mathcal{P}\right)\,d\mathbb{P}(\omega),$$ where $H_{\mu_{\omega}}(\mathcal{P})=-\sum_{A\in\mathcal{P}}\mu_{\omega}(A)\log\mu_{\omega}(A),$ and let $h_{\mu}^{fib}(\mathcal{T})=\sup\{h^{fib}_{\mu}(\mathcal{T};\mathcal{P})\}$ where the supremum is taken over all finite partitions $\mathcal{P}$ of $X.$ If $\mathcal{F}$ is countably generated $(\text{mod }0),$ then $$h_{\mu}(\mathcal{T})=h_{\mu}^{fib}(T)+h_{\mathbb{P}}(\sigma).$$
\end{lemma}

\begin{definition}
    A measurable function $T:X\to X$ on a measure space $(X,\mathcal{B},\mu)$ is said to be nonsingular with respect to $\mu$ if $\mu(T^{-1}(B))=0$ for all $B\in\mathcal{B}$ such that $\mu(B)=0.$
\end{definition}

\begin{definition}[Lasota-Yorke map]\label{LY}
    $T:[0,1]\to[0,1]$ is called a Lasota-Yorke map if there exists a finite subdivision $0=b_{0}<b_{1}<\cdots<b_{n}=1$ such that for each restriction $T_{i}$ of $T$ to any subinterval $(b_{i-1}, b_{i}), i=1,...,n:$
    \begin{enumerate}
        \item $T_{i}$ is a homeomorphism on its image; and \label{homeo}
        \item $T_{i}$ is nonsingular with respect to the Lebesgue measure. \label{nonsingular}
    \end{enumerate}
\end{definition}

\begin{lemma}[Duality relation]\label{duality}
    Let $\mathcal{L}_{T}$ be as in Definition \ref{transfer_operator}. Let $f\in C(\mathbb{T}),$ $g\in L^{\infty}(\mathbb{T}).$ Then $\mathcal{L}_{T}$ satisfies the duality relation $$\int_{\mathbb{T}}f\cdot g\circ T\,dm=\int_{\mathbb{T}}g\cdot\mathcal{L}_{T}f\,dm.$$
\end{lemma}

\begin{definition}[Variation]
For a function $f:[0,1]\to\mathbb{R},$ the variation of $f$ is given by  $$\var\left(f\right):=\sup_{\mathcal{P}}\sum_{j=1}^{n}|f(b_{j})-f(b_{j-1})|$$ where $\mathcal{P}=(0=b_{0}<b_{1}<\cdots<b_{n}=1)$ is a finite partition of $[0,1].$
\end{definition}

\com{The following theorem gives sufficient conditions for the existence of a limiting random invariant measure for a class of random maps.}
\begin{theorem}[\cite{buzzi_edoc}]\label{buzzi}
    Let $\sigma$ be an invertible, ergodic, measure-preserving transformation of a probability space $(\Omega, \mathcal{F}, \mathbb{P}).$ Suppose a random Lasota-Yorke map $\mathcal{T}:=(T_{\omega})_{\omega\in\Omega}$ satisfies \begin{enumerate}
        \item $\omega\mapsto(\inf_{x\in[0,1]\backslash \{b_{0},...,b_{d_{\omega}}\}}|T_{\omega}'(x)|,\var(1/|T_{\omega}'|), d_{\omega}, b_{0},...,b_{d_{\omega}})$ is measurable on $\Omega;$ \label{BB_1}
        \item $\lim_{K\to\infty}\int_{\Omega}\log\min(\inf_{x\in[0,1]\backslash \{b_{0},...,b_{d_{\omega}}\}}|T'(x)|, K)\,d\pw>0;$ \label{BB_2}
        \item $\int_{\Omega}\log^{+}\frac{d_{\omega}}{\inf_{x\in[0,1]\backslash \{b_{0},...,b_{d_{\omega}}\}}|T'(x)|}\,d\pw<\infty;$ \label{BB_3}
        \item $\int_{\Omega}\log^{+}\var\left(\frac{1}{|T_{\omega}'|}\right)\,d\pw<\infty;$ and \label{BB_4}
        \item For every non-trivial $I\subset[0,1]$ and $\mathbb{P}$-a.e. $\omega\in\Omega$ there exists $n_{c}:=n_{c}(\omega)$ such that for all $n\ge n_{c}$ $\mathrm{ess}\inf_{x\in[0,1]}(\mathcal{L}_{\omega}^{(n)}1_{I}(x))>0,$ where $1_{I}$ is the characteristic function of $I$ and $\mathcal{L}_{\omega}^{(n)}:=\mathcal{L}_{\sigma^{n-1}\omega}\circ\cdots\circ\mathcal{L}_{\omega}.$ \label{BB_5}
    \end{enumerate} Then, \begin{enumerate}[label=(\roman*)]
        \item  There is a density $h$ on $\Omega\times [0,1],$ such that   $\int_{0}^{1}h_{\omega}dm=1$ and for $\mathbb{P}$-a.e. $\omega\in\Omega,$   $$\mathcal{L}_{\omega}h_{\omega}=h_{\sigma\omega}.$$ Moreover  $h$ is unique modulo $m$ and $\var(h_{\omega})<\infty$ for $\mathbb{P}$-a.e. $\omega\in\Omega;$ and \label{BB_existence}
        \item There exists $\rho<1$ such that for all $h_{*}:X\to[0,\infty)$ with bounded variation and $||h_{*}||_{1}=1$, for $\mathbb{P}$-a.e. $\omega\in\Omega,$ there is $n_{0}(\omega):=n_{0}$ such that for all $n\ge n_{0}$ $$||\mathcal{L}_{\sigma^{-n}\omega}^{(n)}h_{*}-h_{\omega}||_{\infty}\le\rho^n.$$ \label{BB_diff}
    \end{enumerate}
\end{theorem}

The following theorem is a corollary of \cite[Theorem 6.1]{bogenschutz_1992}.
\begin{theorem}\label{fib_entropy_bogenschutz}
    Let $(\mathcal{T},\sigma)$ be a map cocycle with the family of transfer operators $(\mathcal{L}_{\omega})_{\omega\in\Omega}$ \com{as in Definition \ref{transfer_operator}}. Then $$0=\sup_{\mu\ \mathcal{T}-inv.}\left\{h_{\mu}^{fib}(\mathcal{T})-\int_{\Omega\times X}\log|T_{\omega}'| \,d\mu\right\}.$$
\end{theorem}

\begin{definition}[Blaschke product cocycle] 
Let $\sigma$ be an invertible, ergodic, measure-preserving transformation of a probability space $(\Omega, \mathcal{F}, \mathbb{P}).$ Let $\mathcal{T}=(T_{\omega})_{\omega\in\Omega},$ where
     for each $\omega\in\Omega,$ we let   $T_{\omega}:\hat{\mathbb{C}}\to\hat{\mathbb{C}}$ be given by $$T_{\omega}(z)=\rho_{\omega}\prod_{i=1}^{d_{\omega}}\frac{z-a_{i,\omega}}{1-\bar{a}_{i,\omega}z},$$ where  $n:\Omega\to\mathbb{N},\rho:\Omega\to\mathbb{T}$ are measurable and for each $j\ge1$ such that $\Omega_{j}:=\{\omega\in\Omega:m_{\omega}=j\}$ is non-empty, $a:\Omega_{j}\to D^{j}, D^{j}:=\underbrace{D\times\cdots\times D}_{j \text{ times}}$ is measurable. The Blaschke product cocycle is $(\mathcal{T}, \sigma),$ or simply $\mathcal{T}.$ We let $T_{\omega}^{(n)}=T_{\sigma^{n-1}\omega}\circ\cdots\circ T_{\omega}$ \com{and $T_{\sigma^{-n}\omega}^{(n)}=T_{\sigma^{-1}\omega}\circ\cdots\circ T_{\sigma^{-n}\omega}.$}
\end{definition}
\com{In Theorem \ref{main_thm2}, we consider the family $\mathcal{T}_{\theta}:=(\theta T_{\omega})_{\omega\in\Omega}.$ For all $\omega\in\Omega,$ we let $T_{\omega,\theta}:=\theta T_{\omega}$ and $T_{\sigma^{-n}\omega,\theta}^{(n)}=T_{\sigma^{-1}\omega,\theta}\circ\cdots\circ T_{\sigma^{-n}\omega,\theta}.$ When $\theta=1,$ we will write $\mathcal{T}$ and $T_{\sigma^{-n}\omega}^{(n)}.$}

The following theorem gives a sufficient condition on a Blaschke product cocycle for the existence of measurable $x:\Omega\to D$ such that for $\mathbb{P}$-a.e. $\omega\in\Omega$ and all $z\in D,$ $x_{\omega}=\lim_{n\to\infty}T_{\sigma^{-n}\omega}^{(n)}(z).$ We call $(x_{\omega})_{\omega\in\Omega}$ a random fixed point.
\begin{theorem}[{\cite[Theorem 1(1)]{GTQ}}]\label{gtq_rfp}
    Let $\sigma$ be an invertible ergodic measure-preserving transformation of a probability space $(\Omega,\mathbb{P}).$ Let $r_{T}(R)=\sup_{|z|=R}|T(z)|.$ We denote $r_{\mathcal{T}}(R)=\mathrm{ess}\sup_{\omega\in\Omega}r_{T_{\omega}}(R).$ Let $R<1$ and let $\mathcal{T}=(T_{\omega})_{\omega\in\Omega}$ be a Blaschke product cocycle, depending measurably on $\omega,$ satisfying $r:=r_{\mathcal{T}}(R)<R.$ Then, there exists a measurable map $x:\Omega\to\bar{D}_{r}$ (with $x(\omega)$ written as $x_{\omega}$), such that $T_{\omega}(x_{\omega})=x_{\sigma\omega},$ furthermore, for all $z\in D_{R},$ $T_{\sigma^{-N}\omega}^{(N)}(z):=T_{\sigma^{-1}\omega}\circ\cdots\circ T_{\sigma^{-N}\omega}(z)\to x_{\omega}$ as $N\to\infty.$
\end{theorem}
\com{\begin{remark}
    The hypotheses of Theorem \ref{gtq_rfp} ensure the existence of a random fixed point through uniform contraction $D_{R}$ into a smaller disc $D_{r}.$ In particular, this means that, for all $\omega\in\Omega,$ $\inf_{z\in\mathbb{T}}|T_{\omega}'(z)|\ge\frac{\log r}{\log R}>1$ \cite[Theorem 2]{tischler}. In contrast, Theorem \ref{buzzi_blaschke} requires expansion on average, allowing for some maps that are not only not uniformly expanding, but have an attracting fixed point on the circle (see Example \ref{not_R}).
\end{remark}}

\section{Measure theoretic entropy for an admissible Blaschke product cocycle}\label{entropy_section}

In this section, we obtain a formula for the measure theoretic entropy of an admissible Blaschke product cocycle $\mathcal{T}.$

For $B\subset\mathbb{T},$ let $m(B)$ denote the normalized Lebesgue measure of $B,$ so that $m(\mathbb{T})=1$. 

\begin{definition}[Admissible Blaschke product cocycle]\label{admissible_BP}
    Consider a Blaschke product cocycle $(\mathcal{T},\sigma).$ Let $d_{\omega}:=\deg(T_{\omega}).$ Suppose the following conditions hold:
    \begin{enumerate}
    \item $\omega\mapsto\left(\infzt|T_{\omega}'(z)|, \var\left(\frac{1}{|T_{\omega}'|}\right)\right)$ is measurable on $\Omega;$ \label{ABP_1} 
        \item \com{$\log\inf_{z\in\mathbb{T}}|T_{\omega}'(z)|\in L^{1}(\mathbb{P})\text{ and }$}$\int_{\Omega}\log\infzt|T_{\omega}'(z)|\,d\pw>0;$ \label{ABP_2}
        \item $\int_{\Omega}\log\frac{d_{\omega}}{\infzt |T_{\omega}'(z)|}\,d\pw<\infty;$ and \label{ABP_3}
        \item $\int_{\Omega}\log^{+}\tint\left|\frac{T_{\omega}''}{(T_{\omega}')^2}\right|\,dm\,d\mathbb{P}(\omega)<\infty.$ \label{ABP_4}
    \end{enumerate} Then $(\mathcal{T}, \sigma),$ or simply $\mathcal{T},$ is called an admissible Blaschke product cocycle.
\end{definition}
\com{\begin{remark}
    These conditions ensure that the cocycle, conjugated to $[0,1)$ by $\tilde{T}_{\omega}=Q^{-1}\circ T_{\omega}\circ Q,$ where $Q(x)=e^{2\pi i x},$ satisfies the hypotheses of Theorem \ref{buzzi}. For all $\omega\in\Omega,$ $\lim_{K\to\infty}\log\min(K, \inf_{z\in\mathbb{T}}|T_{\omega}'(z)|)=\log\inf_{z\in\mathbb{T}}|T_{\omega}'(z)|$ and, since $\log\inf_{z\in\mathbb{T}}|T_{\omega}'(z)|$ is integrable (Condition $(2)$), so Theorem \ref{buzzi}$(2)$ holds by dominated convergence. In addition, Condition $(2)$ in this setting implies Theorem \ref{buzzi}$(5)$ (see Lemma \ref{covering}). Conditions $ (1),(3),$ and $(4)$ are equivalent to the corresponding conditions in Theorem \ref{buzzi} (see the proof of Theorem \ref{buzzi_blaschke}$(ii)$). 
\end{remark}}

The proof of the following lemma follows similarly to \cite[Example 6]{a.s.invariance}. 
\begin{lemma}\label{covering}
     Take non-trivial arc $A\subset\mathbb{T}$ such that $m(A)<1.$ Suppose $\mathcal{T}$ is an admissible Blaschke product cocycle. Then, for $\mathbb{P}$-a.e. $\omega\in\Omega,$ there exists $n_{c}(\omega, A):=n_{c}(\omega)$ such that for every $n\ge n_{c}(\omega)$ we have $m(T_{\omega}^{(n)}(A))=1.$
\end{lemma} 

\begin{proof}
    The measure $m$ can be thought of as the normalized arc length measure on the unit circle. Assuming $m(T_{\omega}^{(n)}(A))<1,$ we have $$m(T_{\omega}^{(n)}(A))\ge\infzt|(T_{\omega}^{(n)})'(z)|m(A)\ge m(A)\prod_{k=0}^{n-1}\infzt|T_{\sigma^{k}\omega}'(z)|.$$ By the Birkhoff ergodic theorem, for $\mathbb{P}$-a.e. $\omega\in\Omega,$ $$\lim_{n\to\infty}\frac{1}{n}\sum_{k=0}^{n-1}\infzt|T_{\sigma^{k}\omega}'(z)|=\int_{\Omega}\log\infzt|T_{\omega}'(z)|\,d\pw=:\Lambda>0.$$ We choose $n_{0}:=n_{0}(\omega)$ sufficiently large such that $\frac{1}{n_{0}}\sum_{k=0}^{n_{0}-1}\infzt|T_{\sigma^{k}\omega}'(z)|\ge\Lambda/2.$  Thus, for all $n\ge n_{0},$ we have $$\prod_{k=0}^{n-1}\infzt|T_{\sigma^{k}\omega}'(z)|\ge e^{n\Lambda/2}.$$ Now, taking $n_{c}(\omega)=\max\left(n_{0}(\omega),\lceil-\frac{2}{\Lambda}\log m(A)\rceil\right)$ it follows that $m(T_{\omega}^{(n)}(A))=1$ for $n\ge n_{c}(\omega).$
\end{proof}

\begin{lemma}\label{rim}
    Assume $(\mathcal{T}, \sigma)$ is a Blaschke product cocycle. Suppose there is measurable $x:\Omega\to D$ such that $T_{\omega}(x_{\omega})=x_{\sigma\omega}.$ Suppose $\mu$ is a measure on $\Omega\times X$ with disintegration $\{\mu_{\omega}\}_{\omega\in\Omega}$ and marginal $\mathbb{P}$ with respect to $\Omega$ such that $\frac{d\mu_{\omega}}{dm}:=P_{x_{\omega}}.$ Then $\{\mu_{\omega}\}_{\omega\in\Omega}$ is a $\mathcal{T}-$invariant measure.
\end{lemma}
\begin{proof}
     Take $f\in C(\mathbb{T}).$  We have \begin{align*}
     \int_{\mathbb{T}}f\circ T_{\omega}\,d\mu_{\omega}&=\int_{\mathbb{T}}f\circ T_{\omega}P_{x_{\omega}}\,dm\\&=\tint fP_{T_{\omega}(x_{\omega})}\,dm\\&=\tint fP_{x_{\sigma\omega}}\,dm\\&=\tint f\,d\mu_{\sigma\omega},
     \end{align*} where the second line follows from Theorem \ref{poisson_aut}. 
\end{proof}

We show that the conditions for an admissible Blaschke product cocycle can be identified with the hypotheses of Theorem \ref{buzzi} and use this result to describe the unique absolutely continuous invariant measure for such a cocycle.
\begin{proof}[Proof of Theorem \ref{buzzi_blaschke}\ref{main_thm_diff_bound}]
    Let $(\mathcal{T}, \sigma)$ be an admissible Blaschke product cocycle as in Definition \ref{admissible_BP}.

    Since each Blaschke product $T_{\omega}$ is an $d_{\omega}$-to-$1$ local homeomorphism, it follows that Definition \ref{LY}\ref{homeo} is satisfied. Definition \ref{LY}\ref{nonsingular} is satisfied since any Blaschke product $T$ is a local diffeomorphism and thus nonsingular with respect to Lebesgue \cite{dynamics_of_circle}. 
    
    Next we check the conditions of Theorem \ref{buzzi}. Let $S_{\omega}=Q^{-1}T_{\omega}Q$ as in Definition \ref{blaschke_defn}. First note that $\left|\frac{dS_{\omega}(t)}{dt}\right|=\left|\frac{dT_{\omega}(z)}{dz}\right|.$  
    
    From $\left|\frac{dS_{\omega}(t)}{dt}\right|=\left|\frac{dT_{\omega}(z)}{dz}\right|,$ it follows that $\text{var}\left(\frac{1}{|T_{\omega}'|}\right)=\int_{\mathbb{T}}\left|\frac{T_{\omega}''}{(T_{\omega}')^2}\right|\,dm.$ Thus, Definition \ref{admissible_BP}\ref{ABP_4} is equivalent to Theorem \ref{buzzi}\ref{BB_4}. 
    
    Lastly, Lemma \ref{covering} implies Theorem \ref{buzzi}\ref{BB_5}. Thus, Theorem \ref{buzzi_blaschke}\ref{main_thm_diff_bound} follows from Theorem \ref{buzzi}.

    \end{proof}

\begin{proof}[Proof of Theorem \ref{buzzi_blaschke}\ref{main_thm_rfp}] \begin{description}
\item[\com{Step 1 (Identification with Poisson kernel)}]Take $h^{*}=1.$ From Theorem \ref{buzzi}, it follows that for $\mathbb{P}$-a.e. $\omega\in\Omega$ there exists $h_{\omega}$ of bounded variation such that $\lim_{n\to\infty}\mathcal{L}_{\sigma^{-n}\omega}^{(n)}1=h_{\omega}.$ 

From Lemma \ref{duality} and Lemma \ref{poisson_aut}, we have $$\int_{\bbt}f\cdot\mathcal{L}_{\sigma^{-n}\omega}^{(n)}1\,dm=\int_{\bbt}f\circ T_{\sigma^{-n}\omega}^{(n)}\,dm=\tint fP_{T_{\sigma^{-n}\omega}^{(n)}(0)}\,dm$$ for all $f\in C(\mathbb{T}).$ It follows that $\forall n\in\mathbb{N},$ $\mathcal{L}_{\sigma^{-n}\omega}^{(n)}1=P_{T_{\sigma^{-n}\omega}^{(n)}(0)}$ and thus from Theorem \ref{buzzi}\ref{BB_diff} that $\lim_{n\to\infty}P_{T_{\sigma^{-n}\omega}^{(n)}(0)}=h_{\omega}$ in $L^{\infty}.$

\item[\com{Step 2 (Extraction of a convergent subsequence)}]Now we will show for $\mathbb{P}$-a.e. $\omega\in\Omega$ the existence of the limit $y_{\omega}=\lim_{n\to\infty}T_{\sigma^{-n}\omega}^{(n)}(0),$ and that $h_{\omega}=P_{y_{\omega}}.$   Consider the sequence $T_{\sigma^{-n}\omega}^{(n)}(0)\subset D.$ This must have a convergence subsequence. In particular, there must be a subsequence $(n_{\omega,k}):=(n_{k})$ such that $\lim_{k\to\infty}T_{\sigma^{-n_{k}}\omega}^{(n_{k})}(0)=y_{\omega}\in\bar{D}.$

Suppose that $y_{\omega}\in\mathbb{T},$ and so $|T_{\sigma^{-n}\omega}^{(n)}(0)|\to1.$ From Lemma \ref{l_infty}, it follows that $||h_{\omega}||_{\infty}=\lim_{n\to\infty}||P_{T_{\sigma^{-n}\omega}^{(n)}(0)}||_{\infty}=\infty.$ This can only happen on $\mathbb{P}$ measure 0, since $h_{\omega}$ is of bounded variation for $\mathbb{P}$-a.e. $\omega\in\Omega,$ in particular $||h_{\omega}||_{\infty}<\infty.$ Thus, we must have $y_{\omega}\in D$ for $\mathbb{P}$-a.e. $\omega\in\Omega.$
 \item[\com{Step 3 (Uniqueness of the random fixed point)}]
  If $y_{\omega}\in D,$ then from Lemma \ref{poisson_diff}, $$\lim_{k\to\infty}|T_{\sigma^{-n_{k}}\omega}^{(n_{k})}(0)-y_{\omega}|=0\implies \lim_{k\to\infty}\left|\left|P_{T_{\sigma^{-n_{k}}\omega}^{(n_{k})}(0)}-P_{y_{\omega}}\right|\right|_{\infty}=0.$$ Furthermore, the uniqueness of $h_{\omega}$ gives $\lim_{n\to\infty}T_{\sigma^{-n}\omega}^{(n)}(0)=y_{\omega}.$ For fixed $z\in D,$ taking $h_{*}=P_{z}$ gives $\lim_{n\to\infty}T_{\sigma^{-n}\omega}^{(n)}(z)=y_{\omega},$ and thus \ref{main_thm_rfp} holds. 
 
  We call $(y_{\omega})_{\omega\in\Omega}$ a random fixed point of $\mathcal{T}.$
\item[\com{Step 4 (Construction and invariance of $\mu$)}]
 Lastly, let $\mu$ be the measure with disintegration $\{\mu_{\omega}\}_{\omega\in\Omega}$ with respect to $\mathbb{P},$ where $\frac{d\mu_{\omega}}{dm}=P_{y_{\omega}}.$  It follows from Lemma \ref{rim} that $\{\mu_{\omega}\}_{\omega\in\Omega}$ is $\mathcal{T}$-invariant.

     \end{description}
\end{proof}

We now give the measure theoretic entropy of $\mathcal{T}$ with respect to $\mu.$
\begin{corollary}\label{skew_product_entropy}
    Let $\mathcal{T}$ be an admissible random Blaschke product with $\mu$ its random acim as in Theorem \ref{buzzi_blaschke}. Then, the metric entropy with respect to $\mu$ of $\mathcal{T}$ is given by $$h_{\mu}(\mathcal{T})=h_{\mathbb{P}}(\sigma)+\int_{\Omega}\int_{\mathbb{T}}\log|T'_{\omega}|\,d\mu_{\omega}\,d\mathbb{P}(\omega).$$ 
\end{corollary}
\begin{proof}
From Lemma \ref{skew}, it follows that $$h_{\mu}(\mathcal{T})=h^{fib}_{\mu}(\mathcal{T})+h_{\mathbb{P}}(\sigma).$$ 

By Theorem \ref{fib_entropy_bogenschutz}, $\sup_{\mu_{0}\text{ }\mathcal{T}-inv.}\left(h^{fib}_{\mu_{0}}(\mathcal{T})+\int_{\Omega\times\mathbb{T}}-\log|T'_{\omega}|\,d\mu_{0}\right)=0.$ Thus, with random acim $\mu$ as in Theorem \ref{buzzi_blaschke}, $$h_{\mu}^{fib}(\mathcal{T})=\int_{\mathbb{T}\times\Omega}\log|T'_{\omega}|\,d\mu=\int_{\Omega}\int_{\mathbb{T}}\log|T'_{\omega}|\,d\mu_{\omega}\,d\mathbb{P}(\omega).$$ Thus the claim follows.
\end{proof}

\begin{example}
    Let $\mathcal{T}$ be a Blaschke product cocycle. Let $A_{\omega}=\sum_{i=1}^{d_{\omega}}\frac{1-|a_{i,\omega}|}{1+|a_{i,\omega}|}.$ Suppose $\mathcal{T}$ satisfies the following assumptions:
    \begin{enumerate}
        \item $\int_{\Omega}\log A_{\omega}\,d\pw>0;$ \label{ex_1}
        \item $\int_{\Omega}\log\frac{d_{\omega}}{A_{\omega}}\,d\pw<\infty;$ and \label{ex_2}
        \item $\int_{\Omega}\log\frac{1}{(1-\max_{k=1,...,d_{\omega}}|a_{k,\omega}|)A_{\omega}}\,d\pw<\infty.$ \label{ex_3}
    \end{enumerate} Then, the hypotheses of Theorem \ref{buzzi_blaschke} hold. Indeed, it follows from \cite{Martin_expanding} that $\infzt|T_{\omega}'(z)|\ge\sum_{i=1}^{d_{\omega}}\frac{1-|a_{i,\omega}|}{1+|a_{i,\omega}|},$ and so \ref{ex_1} gives \ref{admissible_BP}\ref{ABP_2}, \ref{ex_2} gives \ref{admissible_BP}\ref{ABP_3}. From applying the triangle inequality, it follows that \ref{ex_3} gives \ref{admissible_BP}\ref{ABP_4}.
    
\end{example}

\section{Average entropy over a family of admissible Blaschke product cocycles}\label{average_section}

Let $\mathcal{T}$ be an admissible Blaschke product cocycle over $\sigma$. For fixed $\theta\in\mathbb{C},$ let $\mathcal{T}_{\theta}:=(\theta T_{\omega})_{\omega\in\Omega}.$ In this section, we obtain a formula for the average metric entropy over $\theta\in\mathbb{T}$ for the family of cocycles $(\mathcal{T}_{\theta})_{\theta\in\mathbb{T}}.$ This allows us to obtain information on the entropy without knowledge of particular orbits, in particular, without needing to find explicitly the invariant measure for each $\mathcal{T}_{\theta}.$

Let $T_{\omega, \theta}:=\theta T_{\omega}.$ The next lemma follows as in the deterministic case in \cite[Proposition 4.3]{expandPujals}. We include the proof for completeness.
\begin{lemma}\label{Leb_rand}
      Let $f:\mathbb{T}\to\mathbb{C}$ be a continuous function.  Then, for all $n\in\mathbb{N}, \omega\in\Omega$, $\int_{\mathbb{T}}\int_{\mathbb{T}} f\circ T_{\sigma^{-n}\omega,\theta}^{(n)} \,dm\,d\theta=\int_{\mathbb{T}}f\,dm.$
\end{lemma}

\begin{proof}
    First, consider the inner integral. Recall that for a continuous function $f:\mathbb{T}\to\mathbb{C},$ $\tilde{f}$ denotes the harmonic extension of $f$ to the disc. We have \begin{align*}
    \int_{\mathbb{T}}f\circ T^{(n)}_{\sigma^{-n}\omega,\theta}\,dm&=\wtildea{f\circ T^{(n)}_{\sigma^{-n}\omega,\theta}}(0)\\&=\tilde{f}\circ T^{(n)}_{\sigma^{-n}\omega,\theta}(0), 
    \end{align*} where the first line follows from Proposition \ref{harmonic} and the second from Lemma \ref{harmonic_holomorphic}. 
    Note that $\theta\mapsto T^{(n)}_{\sigma^{-n}\omega,\theta}(0)$ is analytic in $D$ for each $\omega$ and $n,$ and $T_{\sigma^{-n}\omega,0}^{(n)}(0)=0.$ Since the average value with respect to Lebesgue of a harmonic function on the boundary of a disc is equal to the harmonic function evaluated at the center of the disc ($\theta=0$), it thus follows that \begin{align*}
            \int_{\mathbb{T}}\int_{\mathbb{T}}f\circ T_{\sigma^{-n}\omega,\theta}^{(n)}\,dm\,d\theta&=\int_{\mathbb{T}}\tilde{f}\circ T_{\sigma^{-n}\omega,\theta}^{(n)}(0)d\theta\\&=\Tilde{f}\circ T^{(n)}_{\sigma^{-n}\omega,
        0}(0)\\&=\tilde{f}(0)\\&=\int_{\mathbb{T}}f\,dm.
        \end{align*}
\end{proof}

We now use Theorem \ref{buzzi_blaschke} and Lemma \ref{Leb_rand} to compute the average fibre entropy of $\mathcal{T}_{\theta}$ over $\theta\in\mathbb{T}.$
\begin{proposition}\label{average}
    Let $\mathcal{T}_{1}$ be an admissible Blaschke product cocycle. Let $\{f_{\omega}\}_{\omega\in\Omega}$ be a measurable family of functions, where for each $\omega\in\Omega,$ $f_{\omega}:\mathbb{T}\to\mathbb{C}$ is continuous. For fixed $\theta\in\mathbb{T},$ let $\mu_{\theta}$ be the $\mathcal{T}_{\theta}$-invariant measure as in Theorem \ref{buzzi_blaschke}. Then for $\mathbb{P}$-a.e. $\omega\in\Omega,$ $$\int_{\Omega}\int_{\mathbb{T}}\int_{\mathbb{T}}f_{\omega}\,d\mu_{\omega,\theta}\,d\theta \,d\mathbb{P}(\omega)=\int_{\Omega}\int_{\mathbb{T}}f_{\omega}\,dm\,d\mathbb{P}(\omega),$$ where $m$ denotes the Lebesgue measure.
\end{proposition}
\begin{proof}
    First, note that if $\mathcal{T}_{1}$ satisfies the hypothesis of Theorem \ref{buzzi_blaschke}, then so does $\mathcal{T}_{\theta}$ for every $\theta\in\mathbb{T}.$
    In particular, for $\mathbb{P}$-a.e. $\omega\in\Omega,$ we have the existence of $x_{\omega,\theta}:=\lim_{n\to\infty}T_{\sigma^{-n}\omega,\theta}^{(n)}(0).$ Recall that we can write $\int_{\mathbb{T}}f_{\omega}\circ T_{\sigma^{-n}\omega}^{(n)}\,dm=\int_{\mathbb{T}}f_{\omega}P_{T_{\sigma^{-n}\omega,\theta}^{(n)}(0)}\,dm.$  Let $\tilde{f}_{\omega}$ be the harmonic extension of $f_{\omega}$ as in Proposition \ref{harmonic}. Then, for $\mathbb{P}$-a.e. $\omega\in\Omega,$
    \begin{align*}
    \int_{\mathbb{T}}\int_{\mathbb{T}}f_{\omega}\,d\mu_{\omega,\theta}\,d\theta&=\int_{\mathbb{T}}\int_{\mathbb{T}}f_{\omega}P_{x_{\omega,\theta}}\,dm\,d\theta\\&=\int_{\mathbb{T}}\tilde{f}_{\omega}(x_{\omega,\theta})\,d\theta.\end{align*} From Proposition \ref{harmonic}, it follows that for all $z\in D,$ $|\tilde{f}_{\omega}(z)|$ is bounded by $\sup_{z\in\mathbb{T}}|f_{\omega}(z)|<\infty.$ Thus, by dominated convergence,
    \begin{align*}\int_{\mathbb{T}}\tilde{f}(x_{\omega,\theta})\,d\theta&=\int_{\mathbb{T}}\lim_{n\to\infty}\tilde{f}_{\omega}(T_{\sigma^{-n}\omega,\theta}^{(n)}(0))\,d\theta\\&=\lim_{n\to\infty}\int_{\mathbb{T}}\tilde{f}_{\omega}(T_{\sigma^{-n}\omega,\theta}^{(n)}(0))\,d\theta.\end{align*} Lastly, we have \begin{align*}\lim_{n\to\infty}\int_{\mathbb{T}}\tilde{f}_{\omega}(T_{\sigma^{-n}\omega,\theta}^{(n)}(0))\,d\theta&=\lim_{n\to\infty}\int_{\mathbb{T}}\int_{\mathbb{T}}f_{\omega}\circ T_{\sigma^{-n}\omega,\theta}^{(n)}\,dm \,d\theta\\&=\int_{\mathbb{T}}f_{\omega}\,dm,
    \end{align*} where the first line follows from Proposition \ref{harmonic} and the second from Proposition \ref{Leb_rand}. Integrating both sides with respect to $\mathbb{P},$ the claim follows.
\end{proof}

Now, combining  Corollary \ref{skew_product_entropy} and Proposition \ref{average} we prove Theorem \ref{main_thm2}.
\begin{proof}[Proof of Theorem \ref{main_thm2}]
      
     By Corollary \ref{skew_product_entropy}, the fiber entropy of $\mathcal{T}_{\theta}$ is  $$h^{fib}_{\mu_{\theta}}(\mathcal{T}_{\theta})=\int_{\mathbb{T}}\int_{\mathbb{T}}\log|T'_{\omega}(z)|\,d\mu_{\omega,\theta}\,d\mathbb{P}(\omega),$$ where $d\mu_{\omega,\theta}=P_{x_{\omega,\theta}}dm$ and $(x_{\omega,\theta})_{\omega\in\Omega}$ is the random fixed point of $\mathcal{T}_{\theta}.$ Then the average metric entropy over $\theta\in\mathbb{T}$ is given by  $$\bar{h}(\mathcal{T})=\int_{\mathbb{T}}\int_{\Omega}\int_{\mathbb{T}}\log|T'_{\omega}(z)|\,d\mu_{\omega,\theta}\,d\mathbb{P}(\omega)\,d\theta+h_{\mathbb{P}}(\sigma).$$ Applying Fubini's theorem to this integral, we have $$\bar{h}(\mathcal{T})_=\int_{\Omega}\int_{\mathbb{T}}\int_{\mathbb{T}}\log|T'_{\omega}(z)|\,d\mu_{\omega,\theta}\,d\theta \,d\mathbb{P}(\omega)+h_{\mathbb{P}}(\sigma).$$ Consider $f_{\omega}(z)=\log|T'_{\omega}(z)|$ on $\mathbb{T}$. Since $|T_{\omega}'(z)|>0$ for all $z\in\mathbb{T},$ clearly $f_{\omega}$ is continuous and bounded. Thus, $f_{\omega}$ satisfies the hypotheses of Proposition $\ref{average}.$
     Applying Proposition \ref{average} with $f_{\omega}(z)=\log|T'_{\omega}(z)|$, we have \begin{align*}
        \bar{h}(\mathcal{T})=\int_{\Omega}\int_{\mathbb{T}}\log|T_{\omega}'(z)|\,dm \,d\mathbb{P}(\omega)+h_{\mathbb{P}}(\sigma).
    \end{align*}
\end{proof}
\begin{remark}
    $h_{\mu}^{fib}(\mathcal{T})$ is equal to the Lyapunov exponent of $\mathcal{T}$ with respect to $\mu.$ Theorem \ref{main_thm2} shows that for a family of admissible Blaschke product cocycles $(\mathcal{T}_{\theta})_{\theta\in\mathbb{T}},$ the average Lyapunov exponent of $\mathcal{T}_{\theta}$ with respect to $\mu_{\theta}$ over $\theta\in\mathbb{T}$ is equal to the average logarithmic expansion of $\mathcal{T}$ on $\mathbb{T}.$
\end{remark}

\section{Examples}\label{examples}

In this section we will give an example of an admissible Blaschke product cocycle and compute its average measure theoretic entropy. We will also give an example of an inadmissible Blaschke product cocycle which satisfies the conclusions of Corollary \ref{skew_product_entropy} and Theorem \ref{main_thm2}, showing that the conditions on an admissible Blaschke are sufficient but not necessary for these conclusions to hold.

In the following example, the admissible Blaschke product cocycle is not in the setting of Theorem \ref{gtq_rfp}. In particular, it is expanding on average but not uniformly expanding, and so does not satisfy $r_{\mathcal{T}}(R)<R$ for any $R<1.$
\begin{example}\label{not_R}
    Consider a Blaschke product cocycle $\mathcal{T}:=\mathcal{T}_{1}$ comprised of two maps $T_{0}=z^2$ and $T_{1}=-\left(\frac{z-0.4}{1-0.4z}\right)^2$ with $\mathbb{P}(\omega:T_{\omega}=T_{0})=p=1-\mathbb{P}(\omega:T_{\omega}=T_{1}).$ Figure \ref{fig:interval} shows $S_{1}=Q^{-1}T_{1}Q$ as in Definition \ref{blaschke_defn}. $T_{0}$ is uniformly expanding on $\mathbb{T}$ with $|T_{0}'(z)|=2$ for all $z\in\mathbb{T}.$ $T_{1}$ is not uniformly expanding on $\mathbb{T}$ since $\infzt|T_{1}'(z)|=\frac{6}{7},$ and it has an attracting fixed point at $z=-1.$ $\mathcal{T}$ is an admissible Blaschke product cocycle if $p\log2+(1-p)\log\frac{6}{7}>0,$ that is, $p>\frac{\log(6/7)}{\log(3/7)}\approx0.182.$ Take, for example, $p=0.2.$ Applying Theorem \ref{main_thm2} to the family $(\mathcal{T}_{\theta})_{\theta\in\mathbb{T}}$, we have \begin{align*}
        \bar{h}(\mathcal{T})&=0.2\int_{\mathbb{T}}\log|T_{0}'(z)|\,dm+0.8\int_{\mathbb{T}}\log|T_{1}'(z)|\,dm+h_{\mathbb{P}}(\sigma)\\&=0.2\int_{\mathbb{T}}\log2\,dm+0.8\int_{\mathbb{T}}\log\left(2\frac{1-0.4^2}{|z-0.4|^2}\right)\,dm+h_{\mathbb{P}}(\sigma)\\&\approx0.553664+h_{\mathbb{P}}(\sigma).
    \end{align*} 

    For fixed $\theta=e^{2\pi it},$ let $T_{i}:=T_{i,\theta}, i=0,1.$ Figure \ref{h_theta_av} shows a numerical approximation of $h^{fib}_{\mu_{\theta}}(\mathcal{T}_{\theta})$ over $t\in[0,1)$ for two choices of driving $\sigma_{1}, \sigma_{2}:$ 
    \begin{enumerate}
        \item  $\Omega_{1}=\{0,1\}^{\mathbb{Z}}$ with $\mathbb{P}([0\cdot])=0.2=1-\mathbb{P}([1\cdot]).$ $\sigma_{1}$ is the left shift where $\omega:=[...\omega_{-1}\omega_{0}\omega_{1}...]$ and $T_{\omega}:=T_{\omega_{0}};$ and
        \item  $\Omega_{2}=[0,1)$ and $\sigma_{2}\omega=\omega+\frac{1}{\pi}\text{(mod 1)}.$ If $\omega<0.2$ then $T_{\omega}=T_{0},$ otherwise $T_{\omega}=T_{1}.$
    \end{enumerate}
    
    The numerical approximation of the average over $\theta\in\mathbb{T}$ is $0.5531$ for $(\mathcal{T},\sigma_{1})$ and $0.5533$ for $(\mathcal{T},\sigma_{2})$ over 10,000 steps, an error of approximately $0.10\%$ and $0.07\%$ respectively.

\begin{figure}
    \centering
    \includegraphics[width=0.8\linewidth]{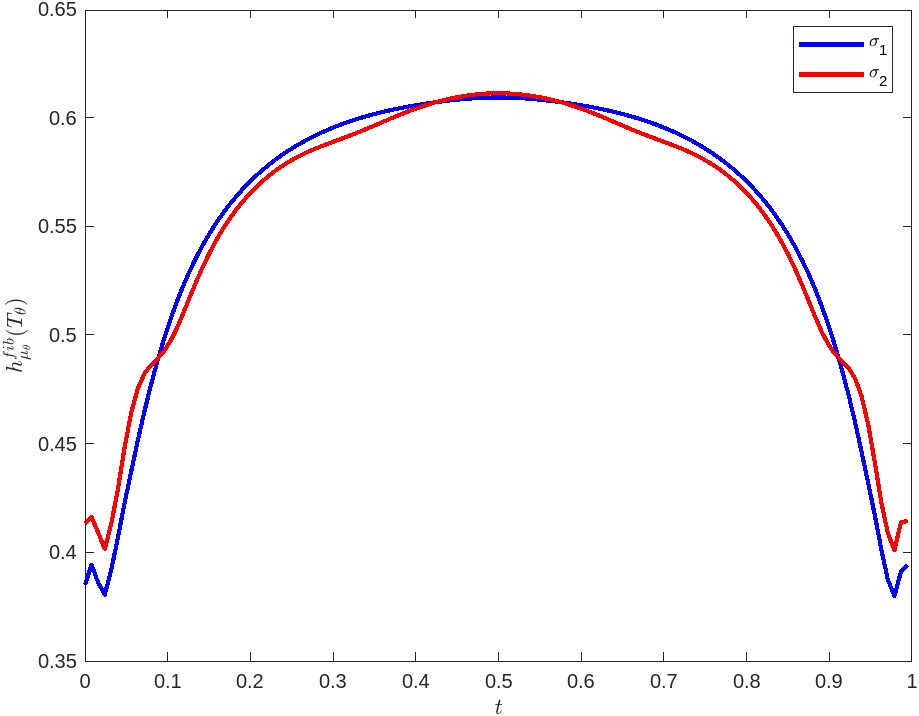}
    \caption{The plot shows $t$ vs $h_{\mu_{\theta}}^{fib}(\mathcal{T}_{\theta}),$ where $\theta=e^{2\pi i t},$  for $(\mathcal{T}_{\theta},\sigma_{1})$ and $(\mathcal{T}_{\theta}, \sigma_{2}).$}
    \label{h_theta_av}
        
\end{figure}

\end{example}

    \com{The admissibility conditions in Definition \ref{admissible_BP} are sufficient but not necessary for the conclusions of Theorem \ref{main_thm2} and Corollary \ref{skew_product_entropy} to hold. In the following example, we give a Blaschke product cocycle which is not admissible, as Definition \ref{admissible_BP}$(3)$ does not hold, yet the average entropy formula still holds by direct computation.}

\begin{example}\label{origin_ex}
    Fix $c>0.$ Let $\Omega=[0,1),$ $\mathbb{P}=m$ and $\sigma\omega=\omega+\alpha\text{ (mod 1)}$ for some fixed $\alpha\not\in\mathbb{Q}.$ For $j\in\mathbb{N},$ let $\Omega_{j}$ be such that $\Omega=\bigcup_{j}\Omega_{j}$ and $\mathbb{P}(\Omega_{j})=\frac{6}{\pi^2(j+1)^2}.$ Let $(\mathcal{T},\sigma)$ be a Blaschke product cocycle consisting of Blaschke products fixing the origin with real zeroes, in particular, for each $\omega\in\Omega_{j}$ let $d_{\omega}=(j+1)^2,$ $a_{i,\omega}\in(-1,1)$ and $$T_{\omega}(z)=\rho_{\omega} z\prod_{i=2}^{d_{\omega}}\frac{z-a_{i,\omega}}{1-a_{i,\omega}z}.$$ Furthermore, for every $j$ such that $j^2-2j\ge c$ assume that $$\inf_{\substack{2\le i\le (j+1)^2\\\omega\in\Omega_{j}}}|a_{i,\omega}|\ge\frac{j^2-2j-c}{j^2-2j+c}.$$  Then, $\mathcal{T}$ is not admissible, but the conclusions of Corollary \ref{skew_product_entropy} and Theorem \ref{main_thm2} hold with $\mu=m\times\mathbb{P}.$ In particular, $$h_{\mu}(\mathcal{T})=\int_{\Omega}\int_{\mathbb{T}}\log|T_{\omega}'|\,dm\,d\mathbb{P}(\omega)+h_{\mathbb{P}}(\sigma),$$ but $$\int_{\Omega}\frac{d_{\omega}}{\infzt|T_{\omega}'(z)|}\,d\pw=\infty.$$
\end{example}
\begin{proof}
     First, note that each $T_{\omega}$ is expanding since $$|T'_{\omega}(z)|=\sum_{i=1}^{d_{\omega}}\frac{1-|a_{i,\omega}|^2}{|z-a_{i,\omega}|^2}\ge1+\sum_{i=2}^{d_{\omega}}\frac{1-|a_{i,\omega}|}{1+|a_{i,\omega}|}>1, \forall z\in\mathbb{T}.$$ 

    For $\mathbb{P}$-a.e. $\omega\in\Omega,$ it follows that $r_{T_{\omega}}(R)<R$ for any choice of $0<R<1$ (see Propositions 4.4, 4.5 and Corollary 4.6 \cite{JPGT}). Applying Theorem \ref{gtq_rfp}, we have for $\mathbb{P}$-a.e. $\omega\in\Omega,$ $$x_{\omega}=\lim_{n\to\infty}T^{(n)}_{\sigma^{-n}\omega}(z)=0, \forall z\in D,$$ and $\mu_{\omega}=m.$ 

    Then, we have $$h_{\mu}^{fib}(\mathcal{T})=\int_{\Omega}\int_{\mathbb{T}}\log|T_{\omega}'|\,dm\,d\mathbb{P}(\omega),$$ and thus $$h_{\mu}(\mathcal{T})=\int_{\Omega}\int_{\mathbb{T}}\log|T_{\omega}'|\,dm\,d\mathbb{P}(\omega)+h_{\mathbb{P}}(\sigma).$$

    Note that since the zeroes are real, for fixed $\omega$ and $1\le i\le d_{\omega},$ we have $\sup_{z\in\mathbb{T}}|z-a_{i,\omega}|=1+|a_{i,\omega}|.$ \com{From Proposition \ref{derivative_bound},} we thus have \begin{align*}
        \infzt|T_{\omega}'(z)|&=1+\infzt\sum_{i=2}^{d_{\omega}}\frac{1-|a_{i,\omega}|^2}{|z-a_{i,\omega}|^2}\\&\le1+(d_{\omega}-1)\sup_{i=2,..., d_{\omega}}\infzt\frac{1-|a_{i,\omega}|^2}{|z-a_{i,\omega}|^2}\\&=1+(d_{\omega}-1)\sup_{i=2,..., d_{\omega}}\frac{1-|a_{i,\omega}|^2}{(1+|a_{i,\omega}|)^2}\\&=1+(d_{\omega}-1)\sup_{i=2,..., d_{\omega}}\frac{1-|a_{i,\omega}|}{1+|a_{i,\omega}|}\\&=1+(d_{\omega}-1)\frac{1-\inf_{i=2,..., d_{\omega}}|a_{i,\omega}|}{1+\inf_{i=2,..., d_{\omega}}|a_{i,\omega}|}.
    \end{align*} If $c> d_{\omega}-1$ then trivially $\infzt|T_{\omega}'(z)|\le d_{\omega}< c+1.$ If $c\le d_{\omega}-1,$ then the hypothesis $\inf_{i=2,..., d_{\omega}}|a_{i,\omega}|\le\frac{d_{\omega}-1-c}{d_{\omega}-1+c}$ gives $$\infzt|T_{\omega}'(z)|\le1+(d_{\omega}-1)\frac{1-\inf_{i=2,..., d_{\omega}}|a_{i,\omega}|}{1+\inf_{i=2,..., d_{\omega}}|a_{i,\omega}|}\le c+1.$$

    Thus, $\mathcal{T}$ is not admissible since \begin{align*}
        \int_{\Omega}\frac{d_{\omega}}{\infzt|T_{\omega}'(z)|}\,d\pw&\ge\frac{1}{c+1}\int_{\Omega}d_{\omega}\,d\pw\\&=\frac{1}{c+1}\sum_{n=1}^{\infty}n\mathbb{P}(d_{\omega}=n)\\&=\frac{1}{c+1}\sum_{j=1}^{\infty}(j+1)^2\mathbb{P}(d_{\omega}=(j+1)^2)\\&=\frac{6}{\pi^2(c+1)}\sum_{j=1}^{\infty}\frac{(j+1)^2}{j^2}=\infty.
    \end{align*}
\com{However, the other conditions for admissibility hold from $$1<\log\inf_{z\in\mathbb{T}}|T_{\omega}'(z)|\le \log(c+1)$$ and the bound on the zeroes.}
Lastly, note that for every $\theta\in\mathbb{T},$ the same hypotheses hold for $\mathcal{T}_{\theta}$. Thus, $h_{\mu}(\mathcal{T}_{\theta})=h_{\mu}(\mathcal{T})$ and clearly $\bar{h}(\mathcal{T})=h_{\mu}(\mathcal{T}),$ but each $\mathcal{T}_{\theta}$ is not admissible.
\end{proof}

\section*{Acknowledgments}
The authors thank Enrique Pujals for encouraging them to pursue
this project. CGT and RO acknowledge support from the Australian Research Council.
CGT acknowledges Lorenzo Díaz and PUC Rio for their hospitality during a visit in 2023,
where this project started. RO acknowledges the Australian Government Research Training
Program for financial support.

Competing interests: The authors declare none.

\printbibliography

@article {JPGT,
    AUTHOR = {Gonz\'{a}lez-Tokman, Cecilia and Peters, Joshua},
     TITLE = {Prevalence of stability for smooth {B}laschke product cocycles
              fixing the origin},
   JOURNAL = {Discrete Contin. Dyn. Syst.},
  FJOURNAL = {Discrete and Continuous Dynamical Systems. Series A},
    VOLUME = {45},
      YEAR = {2025},
    NUMBER = {2},
     PAGES = {585--606},
      ISSN = {1078-0947},
   MRCLASS = {37H30 (30J10 37C30 37F10 37H15)},
  MRNUMBER = {4819619},
       DOI = {10.3934/dcds.2024104},
       URL = {https://doi.org/10.3934/dcds.2024104},
}

@article {GTQ,
    AUTHOR = {Gonz\'{a}lez-Tokman, Cecilia and Quas, Anthony},
     TITLE = {Stability and collapse of the {L}yapunov spectrum for
              {P}erron-{F}robenius operator cocycles},
   JOURNAL = {J. Eur. Math. Soc. (JEMS)},
  FJOURNAL = {Journal of the European Mathematical Society (JEMS)},
    VOLUME = {23},
      YEAR = {2021},
    NUMBER = {10},
     PAGES = {3419--3457},
      ISSN = {1435-9855},
   MRCLASS = {37D25 (30J10 37D35 37E10 37H15)},
  MRNUMBER = {4275477},
MRREVIEWER = {Eugen Mihailescu},
       DOI = {10.4171/jems/1096},
       URL = {https://doi.org/10.4171/jems/1096},
}

@article {expandPujals,
    AUTHOR = {Pujals, Enrique R. and Robert, Leonel and Shub, Michael},
     TITLE = {Expanding maps of the circle rerevisited: positive {L}yapunov
              exponents in a rich family},
   JOURNAL = {Ergodic Theory Dynam. Systems},
  FJOURNAL = {Ergodic Theory and Dynamical Systems},
    VOLUME = {26},
      YEAR = {2006},
    NUMBER = {6},
     PAGES = {1931--1937},
      ISSN = {0143-3857},
   MRCLASS = {37E10 (37A35)},
  MRNUMBER = {2279272},
MRREVIEWER = {Wenxiang Sun},
       DOI = {10.1017/S0143385706000368},
       URL = {https://doi.org/10.1017/S0143385706000368},
}

@book {urbanski,
    AUTHOR = {Mayer, Volker and Skorulski, Bartlomiej and Urbanski, Mariusz},
     TITLE = {Distance expanding random mappings, thermodynamical formalism,
              {G}ibbs measures and fractal geometry},
    SERIES = {Lecture Notes in Mathematics},
    VOLUME = {2036},
 PUBLISHER = {Springer, Heidelberg},
      YEAR = {2011},
     PAGES = {x+112},
      ISBN = {978-3-642-23649-5},
   MRCLASS = {37D35},
  MRNUMBER = {2866474},
MRREVIEWER = {Pei Dong Liu},
       DOI = {10.1007/978-3-642-23650-1},
       URL = {https://doi.org/10.1007/978-3-642-23650-1},
}

@incollection {bogenschutz_AR,
    AUTHOR = {Bogensch\"{u}tz, Thomas and Crauel, Hans},
     TITLE = {The {A}bramov-{R}okhlin formula},
 BOOKTITLE = {Ergodic theory and related topics, {III} ({G}\"{u}strow, 1990)},
    SERIES = {Lecture Notes in Math.},
    VOLUME = {1514},
     PAGES = {32--35},
 PUBLISHER = {Springer, Berlin},
      YEAR = {1992},
   MRCLASS = {28D20},
  MRNUMBER = {1179170},
MRREVIEWER = {Anzelm Iwanik},
       DOI = {10.1007/BFb0097526},
       URL = {https://doi.org/10.1007/BFb0097526},
}

@incollection {og_m.e,
    AUTHOR = {Kolmogorov, A. N.},
     TITLE = {A new metric invariant of transitive dynamical systems and
              automorphisms of {L}ebesgue spaces},
      NOTE = {Topology, ordinary differential equations, dynamical systems},
   JOURNAL = {Trudy Mat. Inst. Steklov.},
  FJOURNAL = {Akademiya Nauk SSSR. Trudy Matematicheskogo Instituta imeni V.
              A. Steklova},
    VOLUME = {169},
      YEAR = {1985},
     PAGES = {94--98, 254},
      ISSN = {0371-9685},
   MRCLASS = {28D20 (58F11)},
  MRNUMBER = {836570},
}

@article {sinai_ent,
    AUTHOR = {Sina\u i, Ja.},
     TITLE = {On the concept of entropy for a dynamic system},
   JOURNAL = {Dokl. Akad. Nauk SSSR},
  FJOURNAL = {Doklady Akademii Nauk SSSR},
    VOLUME = {124},
      YEAR = {1959},
     PAGES = {768--771},
      ISSN = {0002-3264},
   MRCLASS = {28.00 (94.00)},
  MRNUMBER = {103256},
}

@article {pesin_1977,
    AUTHOR = {Pesin, Ja. B.},
     TITLE = {Characteristic {L}yapunov exponents, and smooth ergodic
              theory},
   JOURNAL = {Uspehi Mat. Nauk},
  FJOURNAL = {Akademija Nauk SSSR i Moskovskoe Matemati\v{c}eskoe Ob\v{s}\v{c}estvo.
              Uspehi Matemati\v{c}eskih Nauk},
    VOLUME = {32},
      YEAR = {1977},
    NUMBER = {4(196)},
     PAGES = {55--112, 287},
      ISSN = {0042-1316},
   MRCLASS = {34D05 (28A65 58F10 58F15)},
  MRNUMBER = {466791},
MRREVIEWER = {A. Morimoto},
}

@article {walters_inner,
    AUTHOR = {Walters, Peter},
     TITLE = {Invariant measures and equilibrium states for some mappings
              which expand distances},
   JOURNAL = {Trans. Amer. Math. Soc.},
  FJOURNAL = {Transactions of the American Mathematical Society},
    VOLUME = {236},
      YEAR = {1978},
     PAGES = {121--153},
      ISSN = {0002-9947},
   MRCLASS = {28A65 (58F15)},
  MRNUMBER = {466493},
MRREVIEWER = {B. M. Gurevich},
       DOI = {10.2307/1997777},
       URL = {https://doi.org/10.2307/1997777},
}

@article {bogenschutz_1992,
    AUTHOR = {Bogensch\"{u}tz, Thomas},
     TITLE = {Entropy, pressure, and a variational principle for random
              dynamical systems},
   JOURNAL = {Random Comput. Dynam.},
  FJOURNAL = {Random \& Computational Dynamics},
    VOLUME = {1},
      YEAR = {1992/93},
    NUMBER = {1},
     PAGES = {99--116},
      ISSN = {1061-835X},
   MRCLASS = {28D20 (58F11)},
  MRNUMBER = {1181382},
MRREVIEWER = {Pawe\l  G\'{o}ra},
}

@article {Slipantschuk_2017,
    AUTHOR = {Bandtlow, Oscar F. and Just, Wolfram and Slipantschuk, Julia},
     TITLE = {Spectral structure of transfer operators for expanding circle
              maps},
   JOURNAL = {Ann. Inst. H. Poincar\'{e} C Anal. Non Lin\'{e}aire},
  FJOURNAL = {Annales de l'Institut Henri Poincar\'{e} C. Analyse Non Lin\'{e}aire},
    VOLUME = {34},
      YEAR = {2017},
    NUMBER = {1},
     PAGES = {31--43},
      ISSN = {0294-1449},
   MRCLASS = {37C30 (30H10 37F15)},
  MRNUMBER = {3592677},
MRREVIEWER = {Jo\~{a}o Paulo Almeida},
       DOI = {10.1016/j.anihpc.2015.08.004},
       URL = {https://doi.org/10.1016/j.anihpc.2015.08.004},
}

@article {Martin_expanding,
    AUTHOR = {Martin, N. F. G.},
     TITLE = {On finite {B}laschke products whose restrictions to the unit
              circle are exact endomorphisms},
   JOURNAL = {Bull. London Math. Soc.},
  FJOURNAL = {The Bulletin of the London Mathematical Society},
    VOLUME = {15},
      YEAR = {1983},
    NUMBER = {4},
     PAGES = {343--348},
      ISSN = {0024-6093},
   MRCLASS = {30D50 (28D20)},
  MRNUMBER = {703758},
MRREVIEWER = {Kenneth Stephenson},
       DOI = {10.1112/blms/15.4.343},
       URL = {https://doi.org/10.1112/blms/15.4.343},
}

@proceedings {Blaschke_applications,
     TITLE = {Blaschke products and their applications},
    SERIES = {Fields Institute Communications},
    VOLUME = {65},
 BOOKTITLE = {Proceedings of the conference held at the {U}niversity of
              {T}oronto, {T}oronto, {ON}, {J}uly 25--29, 2011},
    EDITOR = {Mashreghi, Javad and Fricain, Emmanuel},
 PUBLISHER = {Springer, New York; Fields Institute for Research in
              Mathematical Sciences, Toronto, ON},
      YEAR = {2013},
     PAGES = {x+319},
      ISBN = {978-1-4614-5341-3},
   MRCLASS = {30-06 (30H10 30J10)},
  MRNUMBER = {3052284},
       DOI = {10.1007/978-1-4614-5341-3},
       URL = {https://doi.org/10.1007/978-1-4614-5341-3},
}

@book {Blaschke_connect,
    AUTHOR = {Garcia, Stephan Ramon and Mashreghi, Javad and Ross, William
              T.},
     TITLE = {Finite {B}laschke products and their connections},
 PUBLISHER = {Springer, Cham},
      YEAR = {2018},
     PAGES = {xix+328},
      ISBN = {978-3-319-78246-1},
   MRCLASS = {30-02 (47A12 47A30)},
  MRNUMBER = {3793610},
MRREVIEWER = {John R. Akeroyd},
       DOI = {10.1007/978-3-319-78247-8},
       URL = {https://doi.org/10.1007/978-3-319-78247-8},
}

@book {topics_complex,
    AUTHOR = {Andersson, Mats},
     TITLE = {Topics in complex analysis},
    SERIES = {Universitext},
 PUBLISHER = {Springer-Verlag, New York},
      YEAR = {1997},
     PAGES = {viii+157},
      ISBN = {0-387-94754-X},
   MRCLASS = {30-01 (30D35 31-01 46J15 46J20)},
  MRNUMBER = {1419088},
MRREVIEWER = {D. Mitrovi\'{c}},
}

@article {ruelle_entropy,
    AUTHOR = {Ruelle, David},
     TITLE = {An inequality for the entropy of differentiable maps},
   JOURNAL = {Bol. Soc. Brasil. Mat.},
  FJOURNAL = {Boletim da Sociedade Brasileira de Matem\'{a}tica},
    VOLUME = {9},
      YEAR = {1978},
    NUMBER = {1},
     PAGES = {83--87},
      ISSN = {0100-3569},
   MRCLASS = {58F11 (28D20)},
  MRNUMBER = {516310},
MRREVIEWER = {D. Newton},
       DOI = {10.1007/BF02584795},
       URL = {https://doi.org/10.1007/BF02584795},
}

@article {LY1,
    AUTHOR = {Ledrappier, F. and Young, L.-S.},
     TITLE = {The metric entropy of diffeomorphisms. {I}. {C}haracterization
              of measures satisfying {P}esin's entropy formula},
   JOURNAL = {Ann. of Math. (2)},
  FJOURNAL = {Annals of Mathematics. Second Series},
    VOLUME = {122},
      YEAR = {1985},
    NUMBER = {3},
     PAGES = {509--539},
      ISSN = {0003-486X},
   MRCLASS = {58F11 (58F15)},
  MRNUMBER = {819556},
MRREVIEWER = {D. Newton},
       DOI = {10.2307/1971328},
       URL = {https://doi.org/10.2307/1971328},
}

@article {LY2,
    AUTHOR = {Ledrappier, F. and Young, L.-S.},
     TITLE = {The metric entropy of diffeomorphisms. {II}. {R}elations
              between entropy, exponents and dimension},
   JOURNAL = {Ann. of Math. (2)},
  FJOURNAL = {Annals of Mathematics. Second Series},
    VOLUME = {122},
      YEAR = {1985},
    NUMBER = {3},
     PAGES = {540--574},
      ISSN = {0003-486X},
   MRCLASS = {58F11 (58F15)},
  MRNUMBER = {819557},
MRREVIEWER = {D. Newton},
       DOI = {10.2307/1971329},
       URL = {https://doi.org/10.2307/1971329},
}

@article {FGTR_average,
    AUTHOR = {Froyland, Gary and Gonz\'{a}lez-Tokman, Cecilia and Murray, Rua},
     TITLE = {Quenched stochastic stability for eventually
              expanding-on-average random interval map cocycles},
   JOURNAL = {Ergodic Theory Dynam. Systems},
  FJOURNAL = {Ergodic Theory and Dynamical Systems},
    VOLUME = {39},
      YEAR = {2019},
    NUMBER = {10},
     PAGES = {2769--2792},
      ISSN = {0143-3857},
   MRCLASS = {37H10 (28D05 37H15 60B10)},
  MRNUMBER = {4000513},
MRREVIEWER = {Min Zhao},
       DOI = {10.1017/etds.2017.143},
       URL = {https://doi.org/10.1017/etds.2017.143},
}

@article {SU2014,
    AUTHOR = {Simmons, David and Urba\'{n}ski, Mariusz},
     TITLE = {Relative equilibrium states and dimensions of fiberwise
              invariant measures for random distance expanding maps},
   JOURNAL = {Stoch. Dyn.},
  FJOURNAL = {Stochastics and Dynamics},
    VOLUME = {14},
      YEAR = {2014},
    NUMBER = {1},
     PAGES = {1350015, 25},
      ISSN = {0219-4937},
   MRCLASS = {37D35 (37H15)},
  MRNUMBER = {3173965},
       DOI = {10.1142/S0219493713500159},
       URL = {https://doi.org/10.1142/S0219493713500159},
}

@article {AFGV,
    AUTHOR = {Atnip, Jason and Froyland, Gary and Gonz\'{a}lez-Tokman, Cecilia
              and Vaienti, Sandro},
     TITLE = {Equilibrium states for non-transitive random open and closed
              dynamical systems},
   JOURNAL = {Ergodic Theory Dynam. Systems},
  FJOURNAL = {Ergodic Theory and Dynamical Systems},
    VOLUME = {43},
      YEAR = {2023},
    NUMBER = {10},
     PAGES = {3193--3215},
      ISSN = {0143-3857},
   MRCLASS = {37D35 (37E05 37H15)},
  MRNUMBER = {4637150},
       DOI = {10.1017/etds.2022.68},
       URL = {https://doi.org/10.1017/etds.2022.68},
}

@incollection {KK_average,
    AUTHOR = {Khanin, K. and Kifer, Y.},
     TITLE = {Thermodynamic formalism for random transformations and
              statistical mechanics},
 BOOKTITLE = {Sina\u{\i}'s {M}oscow {S}eminar on {D}ynamical {S}ystems},
    SERIES = {Amer. Math. Soc. Transl. Ser. 2},
    VOLUME = {171},
     PAGES = {107--140},
 PUBLISHER = {Amer. Math. Soc., Providence, RI},
      YEAR = {1996},
   MRCLASS = {58F15 (58F11 82B44)},
  MRNUMBER = {1359097},
MRREVIEWER = {Volker Matthias Gundlach},
       DOI = {10.1090/trans2/171/10},
       URL = {https://doi.org/10.1090/trans2/171/10},
}

@article {kifer_expand,
    AUTHOR = {Kifer, Yuri},
     TITLE = {Equilibrium states for random expanding transformations},
   JOURNAL = {Random Comput. Dynam.},
  FJOURNAL = {Random \& Computational Dynamics},
    VOLUME = {1},
      YEAR = {1992/93},
    NUMBER = {1},
     PAGES = {1--31},
      ISSN = {1061-835X},
   MRCLASS = {58F11},
  MRNUMBER = {1181378},
MRREVIEWER = {Ladislav Andrey},
}

@article {AFGV2,
    AUTHOR = {Atnip, Jason and Froyland, Gary and Gonz\'{a}lez-Tokman, Cecilia
              and Vaienti, Sandro},
     TITLE = {Thermodynamic formalism for random weighted covering systems},
   JOURNAL = {Comm. Math. Phys.},
  FJOURNAL = {Communications in Mathematical Physics},
    VOLUME = {386},
      YEAR = {2021},
    NUMBER = {2},
     PAGES = {819--902},
      ISSN = {0010-3616},
   MRCLASS = {37D35 (37E05 37H12)},
  MRNUMBER = {4294282},
MRREVIEWER = {Eugen Mihailescu},
       DOI = {10.1007/s00220-021-04156-1},
       URL = {https://doi.org/10.1007/s00220-021-04156-1},
}

@article {buzzi_SRB,
    AUTHOR = {Buzzi, J\'{e}r\^{o}me},
     TITLE = {Absolutely continuous {S}.{R}.{B}. measures for random
              {L}asota-{Y}orke maps},
   JOURNAL = {Trans. Amer. Math. Soc.},
  FJOURNAL = {Transactions of the American Mathematical Society},
    VOLUME = {352},
      YEAR = {2000},
    NUMBER = {7},
     PAGES = {3289--3303},
      ISSN = {0002-9947},
   MRCLASS = {37D20 (37A25 37E05 37H99)},
  MRNUMBER = {1707698},
MRREVIEWER = {Bernard Schmitt},
       DOI = {10.1090/S0002-9947-00-02607-6},
       URL = {https://doi.org/10.1090/S0002-9947-00-02607-6},
}

@article {buzzi_edoc,
    AUTHOR = {Buzzi, J\'{e}r\^{o}me},
     TITLE = {Exponential decay of correlations for random {L}asota-{Y}orke
              maps},
   JOURNAL = {Comm. Math. Phys.},
  FJOURNAL = {Communications in Mathematical Physics},
    VOLUME = {208},
      YEAR = {1999},
    NUMBER = {1},
     PAGES = {25--54},
      ISSN = {0010-3616},
   MRCLASS = {37H15 (37H05)},
  MRNUMBER = {1729876},
MRREVIEWER = {Bernard Schmitt},
       DOI = {10.1007/s002200050746},
       URL = {https://doi.org/10.1007/s002200050746},
}

@article {a.s.invariance,
    AUTHOR = {Dragi\v{c}evi\'{c}, D. and Hafouta, Y. and Sedro, J.},
     TITLE = {A vector-valued almost sure invariance principle for random
              expanding on average cocycles},
   JOURNAL = {J. Stat. Phys.},
  FJOURNAL = {Journal of Statistical Physics},
    VOLUME = {190},
      YEAR = {2023},
    NUMBER = {3},
     PAGES = {Paper No. 54, 38},
      ISSN = {0022-4715},
   MRCLASS = {37H05 (60F17)},
  MRNUMBER = {4537348},
       DOI = {10.1007/s10955-023-03067-w},
       URL = {https://doi.org/10.1007/s10955-023-03067-w},
}

@article {Blaschke_graph,
    AUTHOR = {Jiang, Yunping},
     TITLE = {Global graph of metric entropy on expanding {B}laschke
              products},
   JOURNAL = {Discrete Contin. Dyn. Syst.},
  FJOURNAL = {Discrete and Continuous Dynamical Systems. Series A},
    VOLUME = {41},
      YEAR = {2021},
    NUMBER = {3},
     PAGES = {1469--1482},
      ISSN = {1078-0947},
   MRCLASS = {37F15 (30J10 37A30 37A35)},
  MRNUMBER = {4201848},
MRREVIEWER = {Yan Mary He},
       DOI = {10.3934/dcds.2020325},
       URL = {https://doi.org/10.3934/dcds.2020325},
}

@article {zero_pos_entropy,
    AUTHOR = {Crovisier, Sylvain and Pujals, Enrique},
     TITLE = {From zero to positive entropy},
   JOURNAL = {Notices Amer. Math. Soc.},
  FJOURNAL = {Notices of the American Mathematical Society},
    VOLUME = {69},
      YEAR = {2022},
    NUMBER = {5},
     PAGES = {748--761},
      ISSN = {0002-9920},
   MRCLASS = {37B40 (37C29 37D15 37E20)},
  MRNUMBER = {4415895},
       DOI = {10.1090/noti2478},
       URL = {https://doi.org/10.1090/noti2478},
}

@book {dynamics_of_circle,
    AUTHOR = {de Faria, Edson and Guarino, Pablo},
     TITLE = {Dynamics of circle mappings},
    SERIES = {IMPA Monographs},
      NOTE = {Second edition [of  4472818]},
 PUBLISHER = {Springer, Cham},
      YEAR = {2024},
     PAGES = {xviii+458},
      ISBN = {978-3-031-67494-5},
   MRCLASS = {37-02 (37C40 37E10 37E20 37F25 37F31 37F50)},
  MRNUMBER = {4824162},
       DOI = {10.1007/978-3-031-67495-2},
       URL = {https://doi.org/10.1007/978-3-031-67495-2},
}

@book {complex_analysis_book,
    AUTHOR = {Stein, Elias M. and Shakarchi, Rami},
     TITLE = {Complex analysis},
    SERIES = {Princeton Lectures in Analysis},
    VOLUME = {2},
 PUBLISHER = {Princeton University Press, Princeton, NJ},
      YEAR = {2003},
     PAGES = {xviii+379},
      ISBN = {0-691-11385-8},
   MRCLASS = {30-01},
  MRNUMBER = {1976398},
MRREVIEWER = {Heinrich Begehr},
}

@article {tischler,
    AUTHOR = {Tischler, David},
     TITLE = {Blaschke products and expanding maps of the circle},
   JOURNAL = {Proc. Amer. Math. Soc.},
  FJOURNAL = {Proceedings of the American Mathematical Society},
    VOLUME = {128},
      YEAR = {2000},
    NUMBER = {2},
     PAGES = {621--622},
      ISSN = {0002-9939},
   MRCLASS = {30D50},
  MRNUMBER = {1641125},
MRREVIEWER = {Walter Bergweiler},
       DOI = {10.1090/S0002-9939-99-05175-8},
       URL = {https://doi.org/10.1090/S0002-9939-99-05175-8},
}

\end{document}